\documentclass[12pt,reqno]{article}

\usepackage[usenames]{color}
\usepackage{amssymb}
\usepackage{graphicx}
\usepackage{amscd}
\usepackage{float}
\usepackage{subfigure,enumerate}

\usepackage{tikz,lipsum,lmodern}
\usepackage[most]{tcolorbox}

\usepackage[]{algorithm2e}

\usepackage{listings}
\definecolor{codegreen}{rgb}{0,0.6,0}
\definecolor{codegray}{rgb}{0.5,0.5,0.5}
\definecolor{codepurple}{rgb}{0.58,0,0.82}
\definecolor{backcolour}{rgb}{0.95,0.95,0.92}

\lstdefinestyle{mystyle}{
  backgroundcolor=\color{backcolour},   commentstyle=\color{codegreen},
  keywordstyle=\color{red},
  numberstyle=\tiny\color{codegray},
  stringstyle=\color{codegreen},
  basicstyle=\footnotesize,
  breakatwhitespace=false,         
  breaklines=true,                 
  captionpos=b,                    
  keepspaces=true,                 
  numbers=left,                    
  numbersep=5pt,                  
  showspaces=false,                
  showstringspaces=false,
  showtabs=false,                  
  tabsize=2
}
\usepackage[colorlinks=true,
linkcolor=webgreen,
filecolor=webbrown,
citecolor=webgreen]{hyperref}

\definecolor{webgreen}{rgb}{0,.5,0}
\definecolor{webbrown}{rgb}{.6,0,0}

\usepackage{color}
\usepackage{fullpage}
\usepackage{float}

\usepackage{graphics,amsmath,amssymb}
\usepackage{amsthm}
\usepackage{amsfonts}
\usepackage{latexsym}
\newcommand{\pprime}{{\prime\prime}}

\newcommand{\seqnum}[1]{\href{http://oeis.org/#1}{\underline{#1}}}

\begin{document}

\begin{center}
\end{center}

\theoremstyle{plain}
\newtheorem{theorem}{Theorem}
\newtheorem{corollary}[theorem]{Corollary}
\newtheorem{lemma}[theorem]{Lemma}
\newtheorem{proposition}[theorem]{Proposition}
\newtheorem{algo}[theorem]{Algorithm}

\theoremstyle{definition}
\newtheorem{definition}[theorem]{Definition}
\newtheorem{example}[theorem]{Example}
\newtheorem{conjecture}[theorem]{Conjecture}
\newtheorem{problem}{Problem}

\theoremstyle{remark}
\newtheorem{remark}[theorem]{Remark}

\begin{center}
\vskip 1cm{\LARGE\bf Chutes and Ladders: on some sequences inspired by 2017 Putnam A1}
\vskip 1cm
Jeremy F.~Alm \\
Department of Mathematics\\
 Lamar University\\
 Beaumont, TX, 77710\\
USA \\
\href{mailto:alm.academic@gmail.com}{\tt alm.academic@gmail.com} \\

\vskip 1cm
Matt Salomone \\
Department of Mathematics\\
Bridgewater State University\\
Bridgewater, MA, 02325\\
USA \\
\href{mailto:msalomone@bridgew.edu }{\tt msalomone@bridgew.edu} \\

\end{center}

\begin{abstract}
  The first problem of the 2017 Putnam competition was to characterize a set of natural numbers closed  under both the square-root map $n^2 \mapsto n$ and the ``add 5 and square'' map $ n \mapsto (n+5)^2$.  We reframe this as a problem on an infinite directed graph, using this framing both to generalize the problem and its solution, as well as to determine the first appearance of each number in this set under a row-wise algorithm that outputs all its elements.
\end{abstract}

\section{Problem A1 from the 2017 Putnam Exam}

The following problem appeared on the 2017 William Lowell Putnam Mathematical Competition \cite{putnam_2017}.
\begin{itemize}
\item[A1] 
Let $S$ be the smallest set of positive integers such that
\begin{enumerate}
\item[(a)]
$2$ is in $S$,
\item[(b)]
$n$ is in $S$ whenever $n^2$ is in $S$, and
\item[(c)]
$(n+5)^2$ is in $S$ whenever $n$ is in $S$.
\end{enumerate}
Which positive integers are not in $S$?

(The set $S$ is ``smallest'' in the sense that $S$ is contained in any other such set.)

\end{itemize}

The solution, which was discussed on Twitter, MathOverFlow, etc., is that all positive integers are in $S$ \emph{except} 1 and the multiples of 5. 

Inspired by this problem, we define the following triangle of integers, as \seqnum{A296142}.
\begin{definition}\label{def:putnamtriangle}
For each $i\geq 1$, define $R_i$ to be the set of integers defined recursively as follows.
\begin{itemize}
    \item $R_1 = \{2\}.$
    \item For each $k\geq 1$, if $x\in R_k$ then $(x+5)^2 \in R_{k+1}$.
    \item For each $k\geq 1$, if $x^2\in R_k$ then $x\in R_{k+1}$.
\end{itemize}
\end{definition} 

The solution of Problem A1 can then be interpreted to say that, for all integers $x\geq 2$, if $x\not\equiv 0 \pmod{5}$, then there exists an $i\geq 0$ for which $x\in R_i$.

More than the existence of such an $i$, however, we pose the question: how does $i$ vary with $x$? In other words, in which row does each integer $x \geq 2$ make its first appearance? The first five rows $R_i$ are shown in Table \ref{tab:sixrows} below.

\begin{table}[htb] \label{tab:sixrows}
\renewcommand\arraystretch{1.25}
    \centering\begin{tabular}{|c|*{5}{r}|}\hline
    $i$ & \multicolumn{5}{|l|}{Row $R_i$} \\ \hline
    1 & 2 &&&&\\ 
    2 & 49&&&&\\
    3 & 7 & 2916&&&\\
    4 & 54 & 144 & 8532241&&\\
    5 & 12 & 2921 & 3481 & 22201 & 72799221804516\\
    \hline\end{tabular}
    \caption{The first five rows $R_1, \ldots, R_5$ of the triangle of integers.}
\end{table}

While the solution of Problem A1 guarantees the appearance of all $x \geq 2, x\not\equiv 0\pmod{5}$ in the sequence of rows, it does not provide insight into {\it when} each such $x$ first appears. In this paper we will develop an algorithm that exhibits the first appearance of each such $x$ up to $x=99$. The techniques herein may be adapted to extend the algorithm to arbitrarily large $x$. We conclude by determining ``why'' Problem A1 has such a maximalist solution (avoiding only the multiples of $5$), and which generalizations of A1 admit analogous solutions.

In particular, we establish in Section \ref{sec:general} that Problem A1, generalized with $R_1 = \{r\}$ and $(x+5)^2$ replaced by $(x+d)^e$, admits an analogously maximal solution whenever $d$ is prime, $r$ is nonzero modulo $d$, and all of the prime factors of $d-1$ are prime factors of $e$. In many cases these conditions are satisfied only in rare instances: for example, the squaring map ($e=2$) has maximal solutions only in the cases where $d$ is a Fermat prime number!

The key insight that will power our argument is the observation that, modulo $5$, this problem reduces to the dynamics of the squaring map $x \mapsto x^2$ on the ring $\mathbb{F}_5^\times$ of units modulo 5. Such power-map dynamics have been well studied: they generally take the form of regular cycles with attached trees, and generically have two or more connected components. Some of the more general results regarding higher-power maps are discussed in Section \ref{sec:general}. 

In the squaring map case,  Chass\'e \cite{Chasse} and Rogers \cite{rogers_1996} independently have characterized the structure of the connected components modulo a prime $p$. In addition, others such as Vasiga \&\ Shallit \cite{Shallit} have extended these results to affine squaring maps of the form $x \mapsto x^2 + c$, of which the case $c=-2$ is exceptional. 

Pomerance and Shparlinski \cite{Pomerance} show that for every integer exponent $e>1$,  there are infinitely many primes $p$ such that the graph of the map $x \mapsto x^e$ has at least $p^{5/12}$ cycles. This shows that even for exponents larger than 2, the graphs of the power map tend not to be connected.

\section{Why This is Hard To Do}\label{sec:whyhard}

After solving Problem A1, one might ask, ``I know that 3 is in $S$, but what's the quickest way from 2 to 3 using rules (b) and (c)?'' In other words, what is the least $i$ such that $3\in R_i$? The difficulty we found in answering this question forms the motivation for the present paper. 

The ``obvious'' approach is simply to compute the rows brute-force.  Well, why not?  But in fact, saying ``Just compute $R_i$'' turned out to be similar to saying ``Let $p_1^{\alpha_1}\dots p_k^{\alpha_k}$ be the prime factorization of $N$''---if one wants the actual factorization, easier said than done.  

The main obstacle is that repeated squaring quickly produces very large integers.  Not worrying too much about the ``+5'' in (c), row $R_i$ contains an integer roughly the size of $2^{2^i}$.  Since during our computer explorations it seemed that it would suffice to examine the first 150 or so rows, let's think about  $2^{2^{150}}$,  has on the order of $10^{44}$ decimal digits.   This amounts to about $10^{45}$ binary digits, which would require $10^{44}$ (or 100 tredecillibytes) of storage. For comparison, in 2025 it is estimated that the total data storage capacity on the planet will only be approximately 200 zettabytes ($2\times 10^{24}$ bytes). See, e.g., \cite{Freeze_2021}. 

While trying to brute-force the rows, the first author's computer would grind to a halt somewhere around row 50.  A band-aid solution was to modify the code so that the ``add five and square'' rule was not applied to any integer that exceeded $10^{20}$. This allowed computation of partial rows, which allowed us to discover a path from 2 to 3 that put 3 in row 104. This implies that the first appearance of 3 is in row 104 \emph{or earlier}.  There could, it would seem, be a shorter path that involved huge integers and lots of applications of (b), i.e., repeated square-rooting. How can we rule out such a path?  As we will see, paths that begin and end in small integers that pass through huge integers are necessarily long.  A key observation that heuristically ``makes it all work''  is that two ``up-steps'' via (c) followed by two ``down-steps'' via (b) is not possible---see Theorem \ref{thm:uudd}. 

\begin{table}[htb]\label{tab:oeis}
\centering\begin{tabular}{r|*{12}{c}}\hline\hline
$x$ &2&3&4&6&7&8&9&11&12&13&14&16 \\
$a(x)$ & 1&104&122&130&3&9&103&119&5&11&105&121\\ \hline
\multicolumn{13}{c}{\; } \\ \hline
$x$ &17&18&19&21&22&23&24&26&27&28&29&31 \\
$a(x)$ & 7&13&107&123&9&15&109&125&11&17&111&127 \\ \hline
\multicolumn{13}{c}{\; } \\ \hline
$x$ & 32&33&34&36&37&38&39&41&42&43&44&46 \\
$a(x)$ & 13 & 19 & 113& 129&15&21&115&131&17&23&117&133 \\ \hline
\multicolumn{13}{c}{\; } \\ \hline
$x$ & 47&48&49&51&52&53&54&56&57&58&59&61\\
$a(x)$ & 19 & 25 & 2 & 135 & 21 & 27 & 4 & 92 & 23&29&6&94\\ \hline
\multicolumn{13}{c}{\; } \\ \hline
$x$ &62&63&64&66&67&68&69&71&72&73&74&76\\
$a(x)$ & 25 &31&8&96& 27 &33&10&98& 29 &35&12&100 \\ \hline
\multicolumn{13}{c}{\; } \\ \hline
$x$ &77&78&79&81&82&83&84&86&87&88&89&91\\
$a(x)$& 31 &37&14&102& 33 &39&16&104& 35 &41&18&106 \\ \hline
\multicolumn{13}{c}{\; } \\ \hline
$x$ &92&93&94&96&97&98&99\\
$a(x)$ & 37 &43&20&108& 39 &45&22\\ \hline\hline
\end{tabular}\caption{The first appearance row $a(x)=\min\{i\colon x\in R_i\}$ for each $x \in \{2,\ldots,100\}$ which is not a multiple of 5. See \seqnum{A366552}.}
\end{table}

\section{Chutes and Ladders on Squaring Graphs}

Denote by $S$ the set of all positive integers appearing in at least one row $R_i$. The published solution of Problem A1 is that \begin{equation*} S = \{ n \in \mathbb{N} \colon n\geq 2, n\not\equiv 0 \pmod{5} \}. \end{equation*}

In order to determine with certainty the \emph{first} appearance of a given $x\in S$, we will interpret $S$ as the vertex set of a directed graph $\Gamma_5$. Paths on this graph correspond to appearances of integers on successive rows.  For example, this sequence is a path in $\Gamma_5$:

\[
    4, 81, 9, 3, 64, 8, 169, 13, 324, 18,\ldots 
\]

\begin{definition}\label{def:gammap}
    Let $\Gamma_5$ denote a graph with the following data.

    The vertices of $\Gamma_5$ are the elements of the set $S$ of all integers appearing in at least one row $R_i$.

    The edges of $\Gamma_5$ are drawn from each integer on a given row $R_i$ to either one or two integers on the successive row $R_{i+1}$. The edge set $E(\Gamma_5) = U(\Gamma_5)\cup D(\Gamma_5)$ consists of two types of edges as follows:
    \begin{align*}
        U(\Gamma_5) &= \left\{ \left( n, (n+5)^2 \right) \colon n \in S \right\}\\
        D(\Gamma_5) &= \left\{ (n^2, n) \colon n \in S \right\}
    \end{align*}
\end{definition}
We will call the edges in $U(\Gamma_5)$ ``up-edges'' and the edges in $D(\Gamma_5)$ ``down-edges.'' In this graph the outdegree of each vertex is $2$ for every perfect square in $S$ (possessing both an outgoing up- and down- edge) and $1$ otherwise (possessing only an outgoing up-edge).

By construction, a positive integer $x$ appears in at least one row $R_i$ if, and only if, there exists a path from $2$ to $x$ in the graph $\Gamma_5$. 

The structure of the infinite graph $\Gamma_5$ is best understood via its quotient, replacing the integers in $S$ with their residue classes modulo $5$. The quotient graph $G_5 = \Gamma_5 / \sim$ has \begin{equation*} V(G_5) = \left\{ [1], [2], [3], [4]\right\} \quad \text{and} \quad E(G_5) = \left\{ ([n],[n]^2) \colon n \in S \right\} \cup \left\{ ([n]^2,[n]) \colon n \in S \right\}. \end{equation*}
The graph $G_5$ is precisely the \emph{squaring graph modulo $5$}. Rogers in \cite{rogers_1996} elucidates the structures of all squaring graphs modulo primes; we will draw on some of these more general results in Section \ref{sec:general}.

\begin{figure}
\includegraphics[width=\linewidth]{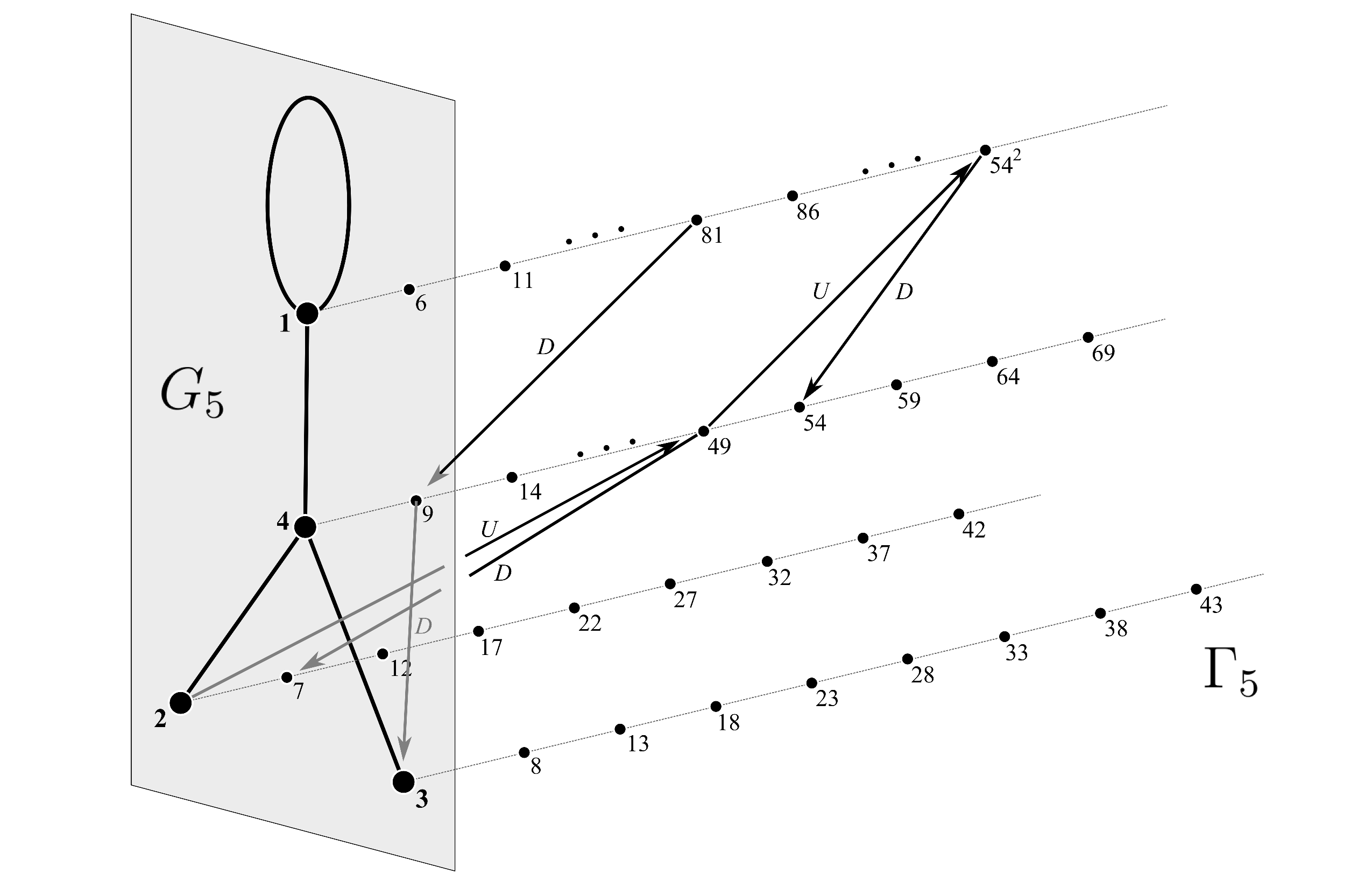}
    \caption{The directed graph $\Gamma_5$ with a selection of its up- and down-edges, and its quotient $G_5$, the (undirected) graph of the squaring map modulo $5$. The dashed ``belts'' can be viewed as $UD$-edges in $\Gamma_5$.}\label{fig:graph}
\end{figure}

We can best understand paths in $\Gamma_5$ according to their corresponding paths in the quotient $G_5$, i.e., by keeping track of the residue class modulo $5$ of the vertices visited along the path. That inspires the central metaphor of this paper.

\begin{definition}
    A \emph{conveyor belt} (or ``\emph{belt}'') is a path in $\Gamma_5$ of the form \[ UD\, UD\, \ldots \, UD.\]
\end{definition}
Every $x\in S$ has a path $UD$ originating from $x$, namely, $x \stackrel{U}\to (x+5)^2 \stackrel{D}\to x+5$. Conveyor belt paths necessarily preserve the residue class modulo 5 (i.e., they begin and end on the same vertex of the squaring graph). These may be seen as the dashed-line paths in Figure \ref{fig:graph}.

Any path in $\Gamma_5$ can be subdivided at its conveyor belt subpaths. The subpaths in between conveyor belts have two types, according to the following result.

\begin{theorem}[Chutes-and-ladders decomposition]\label{thm:decomposition}
    Every path in $\Gamma_5$ has a unique decomposition into subpaths of the following three forms:
    \begin{enumerate}
        \item Conveyor belt subpaths of the form $(UD)^i$ for some $i\geq 1$;
        \item \emph{Ladder} subpaths of the form $U^i$ for some $i\geq 1$; and
        \item \emph{Chute} subpaths of the form $D^i$ for some $i \geq 1$.
    \end{enumerate}
\end{theorem}

\begin{proof}
    Let $P$ be a path in $\Gamma_5$. We begin by identifying all conveyor belt subpaths $(UD)^i$ for $i\geq 1$, noting that two or more consecutive conveyor belts are equivalent to a single conveyor belt. By identifying these belts, $P$ may be decomposed into the form \begin{equation*} P = P^\prime \, C_1\, P_1\, C_2\, \ldots\, P_j C_j\, P^{\prime\prime}, \end{equation*}
    where $C_1,\ldots,C_j$ are conveyor belts and $P^\prime,P_1,\ldots,P_j,P^{\prime\prime}$ are each subpaths which do not contain $(UD)$ as a subpath, with $P_1,\ldots,P_j$ all nonempty.

    Let $P_i$ be a nonempty subpath in $ \left\{ P^\prime,P_1,\ldots,P_j,P^{\prime\prime} \right\}$. If $P_i$ begins with an up-edge $U$, then since by assumption it does not contain $(UD)$ it must consist entirely of up-edges. In this case $P_i$ is a ladder subpath.

    If $P_i$ begins with a down-edge $D$, there are two options. If it consists entirely of down-edges it is a chute subpath. If it contains one up-edge, then since it does not contain $(UD)$ it cannot contain any subsequent down-edges. In this case $P_i = D^n\, U^m$ consists of a chute followed by a ladder.
\end{proof}
At the heart of the chutes-and-ladders analogy is that ladders ``climb higher'' on the squaring graphs while chutes ``fall down lower''. This is evident in Figure \ref{fig:graph}, for example, in the ladder $2\stackrel{U}\to 49 \stackrel{U}\to 54^2$ and the chute $81\stackrel{D}\to 9 \stackrel{D}\to 3.$

We conclude this section by noting that in $\Gamma_5$, conveyor belt and ladder paths visit vertices in a strictly increasing fashion (unless otherwise stated, we are implicitly neglecting the intermediate vertex of each $UD$ belt step, so that each such step begins at $x$ and ends at $x+5$). Meanwhile, chute paths -- since they represent iterated square-roots -- visit vertices in strictly decreasing order along their length.

This observation allows for an alternative proof of Problem A1.

\begin{theorem}\label{thm:connected}
    The graph $\Gamma_5$ is connected.
\end{theorem}

\begin{proof}
    Let $x,y \in S$ be chosen arbitrarily; we will show that a path from $x$ to $y$ exists in $\Gamma_5$.

    From $x$ we may ``climb the ladder'' two steps to $w = UU(x) = \left((x+5)^2+5\right)^2$. In the squaring graph $G_5$, ascending two levels necessarily arrives at the residue $[1]$.

    Likewise, from $y$ we may ``climb the chute backward'' to obtain a strictly increasing sequence $y, y^2, y^4, y^8, y^{16},\ldots$ of integers each of which has an outbound chute ending at $y$. Again, in the squaring graph $G_5$, each of these (except possibly $y,y^2$) projects to the residue $[1]$. Let $z$ be the first member of this increasing sequence which is congruent to $1$ modulo $5$ and for which $z \geq w$. Denote by $C$ the chute path from $z$ to $y$.

    Since $w \leq z$ and both are congruent modulo $5$, there exists a conveyor belt path $B$ from $w$ to $z$.

    Thus there exists a path of the form $P=UU\, B\,C$ originating at $x$ and ending at $y$, completing the proof.
\end{proof}

\begin{corollary}
    Every integer $x$ for which $x \geq 2$ and $x\not\equiv 0 \pmod 5$ appears in at least one row $R_i$.
\end{corollary}

\begin{proof}
    This follows from the existence in particular of a path from $2$ to $x$. If the path contains $\ell$ edges, this implies $x$ appears in row $R_{\ell+1}$.
\end{proof}

The existence of such a path settles the Putnam problem by guaranteeing \emph{that} each integer in $S$ appears in some row, but does not settle the question of \emph{when} each such integer first appears, i.e., the least $i$ for which $x \in R_i$. This is a question not of mere connectedness of $\Gamma_5$ but of shortest paths in $\Gamma_5$, to which we turn in the next section.

\section{Square-to-Square Belts}

While every path in $\Gamma_5$ has a unique decomposition into belts, chutes, and ladders, not every concatenation of belts, chutes, and ladders is realized by a path in $\Gamma_5$. This is a consequence of the fact that while every $x\in S$ has a belt path and a ladder path originating at $x$, \emph{only when $x$ is a perfect square} does it also have an outbound chute path. 

In this section we will explore the limitations this imposes on paths in $\Gamma_5$, with particular attention to the lengths of belt paths that join ladders to chutes. These lengths will be an important tool for bounding from below the lengths of paths from $2$ to $x$, i.e., for guaranteeing the nonappearance of $x$ before a given row $R_i$.

We first show that a ladder subpath may not be immediately followed by a chute. Since consecutive ladders (respectively, chutes) are equivalent to a single one, this result shows that every ladder must be followed by a belt, and every chute must be preceded by a belt.

\begin{theorem}\label{thm:uudd}
    No path in $\Gamma_5$ contains $UUDD$ as a subpath.
\end{theorem}

\begin{proof}
    Suppose to the contrary that for some $x \in S$ there exists a path $UUDD$ originating at $x$. Denote the vertices along this path by \begin{equation*} x \stackrel{U}\to a \stackrel{U}\to z \stackrel{D} \to b \stackrel{D}\to y.\end{equation*}

    By definition of up- and down-edges we can characterize the integer $z$ in two ways: \begin{equation*} z = UU(x) = \left( (x+5)^2 + 5 \right)^2 \quad \text{and} \quad DD(z) = y, \text{ i.e., } z = y^4. \end{equation*}

    We therefore have, since $x,y,z > 0$,
    \begin{align*} y^4 &= \left( (x+5)^2 + 5\right)^2, \text{ and thus}\\ y^2 - (x+5)^2 &= 5.\end{align*}
    In particular this shows $y > x+5$ so that $y^2$ is no smaller than the next perfect square greater than $(x+5)^2$, i.e., $y^2$ is at least $(x+6)^2$. So we also have, again since $x > 0$,
    \begin{align*}y^2 - (x+5)^2 &\geq (x+6)^2 - (x+5)^2\\ &= 2x+11 > 12,\end{align*}
    which is a contradiction. Thus $UUDD$ is an impermissible path in $\Gamma_5$.    
\end{proof}

\begin{corollary}\label{cor:ladderchute}
    Every ladder subpath in $\Gamma_5$ is followed by a conveyor belt, and every chute subpath is preceded by a conveyor belt.
\end{corollary}

\begin{proof}
    Suppose a path $P=U^iD^j$ exists in $\Gamma_5$. We will show that if the $U^i$ form a ladder, the $D^j$ cannot be a chute, and vice versa.

    If $i \geq 2$ then we must have $j=1$ since otherwise $P$ would contain $UUDD$ as a subpath. But then $P=U^iD = U^{i-1}\, UD$ is a ladder followed by a belt and not a chute.

    For the same reason, if $j \geq 2$ then we must have $i=1$ and $P = UD^j = UD\, D^{j-1}$ is a belt (not ladder) followed by a chute.
\end{proof}

The proof of Theorem \ref{thm:uudd} relied upon the separation between adjacent perfect squares. These separations will be the key to bounding from below the lengths of shortest paths in $\Gamma_5$, as follows.

\begin{definition}
    A conveyor belt path $P=(UD)^i$ is called a \emph{square-to-square} belt if it both begins and ends at a perfect square.
\end{definition}

By construction of up- and down-edges, the belts that follow ladders always begin at a perfect square and the belts that precede chutes always end at a perfect square. It is therefore only the square-to-square belts which are capable of both following a ladder and preceding a chute, i.e., if a ladder-belt-chute subpath of the form $L\, B\, C$ exists, then $B$ must be a square-to-square belt. In light of Theorem \ref{thm:uudd} a stronger result in fact holds.

\begin{corollary}\label{cor:squaretosquare}
    Let $P$ be a path in $\Gamma_5$. Then 

    \begin{enumerate}\renewcommand\labelenumi{(\alph{enumi})}
        \item If $P= P^\prime\, L\, P^\pprime$ and $P^\pprime$ contains no ladders and at least one chute, then $P^\pprime$ begins with a square-to-square conveyor belt.

        \item If $P = P^\prime\, C\, P^\pprime$ and $P^\prime$ contains no chutes and at least one ladder, then $P^\prime$ ends with a square-to-square conveyor belt.

        \item If $P = L\, P^\prime\, C$ where $L$ is a ladder and $C$ is a chute, then $P^\prime$ contains a square-to-square conveyor belt subpath.
    \end{enumerate}
    
\end{corollary}

\begin{proof}
    To establish (a), observe that by Corollary \ref{cor:ladderchute}, $P^\pprime$ necessarily begins with a conveyor belt: $P^\pprime = BQ$. Since $P^\pprime$ follows a ladder, $B$ begins at a perfect square. 
    
    But since $Q$ contains no ladders and at least one chute, the belt $B$ must be followed by a chute and therefore must also end at a perfect square, making it a square-to-square belt.

    The proof of (b) is analogous. Since it precedes a chute, $P^\prime$ necessarily ends with a conveyor belt: $P^\prime = RB$ and $B$ ends in a perfect square. 

    But $R$ contains no chutes and at least one ladder, so the belt $B$ must be preceded by a ladder and therefore also begin in a perfect square.

    Claim (c) follows by replacing, without loss of generality, $P^\prime$ in $P = L\, P^\prime\, C$ with that subpath of $P^\prime$ which follows the final ladder of $P$ (therefore contains no ladders) and precedes the final chute (therefore contains no chutes).
\end{proof}

The square-to-square conveyor belt paths are valuable for determining upper bounds for vertices visited along paths in $\Gamma_5$. This is due to the fact that the separation between adjacent perfect squares increases monotonically with the size of the squares themselves.

Note first that all perfect squares modulo $5$ are congruent to $0$, $1$, or $4$, so that any square-to-square conveyor belt in $\Gamma_5$ will project either to $[1]$ or to $[4]$ on the squaring graph $G_5$.

\begin{lemma}\label{lem:maxbelt}
    Let $B$ be a square-to-square conveyor belt path in $\Gamma_5$, and suppose $B$ consists of fewer than $150$ steps, that is, fewer than 75 successive $(UD)$.

    \begin{enumerate}\renewcommand\labelenumi{(\alph{enumi})}
        \item If $B$ projects to $[1]$ on the squaring graph, its endpoint is no larger than $91^2$.
        \item If $B$ projects to $[4]$ on the squaring graph, its endpoint is no larger than $183^2$.
    \end{enumerate}
\end{lemma}

\begin{proof}
    We lose no generality by assuming that $B$ begins and ends at \emph{adjacent} squares, i.e., that it does not visit any perfect squares other than its beginning and end points. If it does, we can replace $B$ with a square-to-square belt having the same endpoint but originating at the next least adjacent square in its residue class modulo $5$.

    For (a), assume that $B$ begins and ends at perfect squares congruent to $1$ modulo $5$. Since the square roots of $1$ modulo $5$ are $1$ and $4$, if $B$ begins at $m^2$ and ends at $n^2$ then either $n-m = 2$ or $n-m=3$.

    If $n-m=2$ then $B$ ends at $n^2$ for some $n\equiv 1\pmod{5}$ and has the form \begin{equation*} (n-2)^2 \stackrel{UD}\to (n-2)^2 + 5 \stackrel{UD}\to (n-2)^2 + 10 \stackrel{UD}\to\cdots\stackrel{UD}\to n^2-5 \stackrel{UD}\to n^2.\end{equation*}
    This path consists of $\frac25\left( n^2 - (n-2)^2 \right) = \frac85(n-1)$ steps. This is less than 150 for $n < 95$, for which the largest such $n$ congruent to $1$ modulo $5$ is $n=91$.

    If $n-m=3$ then $B$ ends at $n^2$ for some $n\equiv 4\pmod{5}$ and consists of $\frac25\left(n^2-(n-3)^2\right) = \frac65(2n-3)$ steps. This is less than 150 for $n < 64$, so the largest possible $n$ congruent to $4$ modulo $5$ is $n=59$, which is a lesser endpoint than in the case above. This confirms claim (a).

    For (b), assume that $B$ begins and ends at perfect squares congruent to $4$ modulo $5$. The mod-$5$ square roots of $4$ are $2$ and $3$, so if $B$ begins at $m^2$ and ends at $n^2$ then we have either $n-m=1$ or $n-m=4$.

    We have $n-m=1$ when $n\equiv 3\pmod 5$ and in this case $B$ consists of $\frac25\left( n^2-(n-1)^2 \right) = \frac25(2n-1)$ steps. This is less than $150$ for $n < 188$, i.e., is maximized when $n=183$.

    When $n-m=4$, i.e., $n\equiv 2\pmod 5$, the path $B$ consists of $\frac25\left(n^2-(n-4)^2\right) = \frac{16}{5}(n-2)$ steps, and this is less than 150 for $n < 49$. In this case the largest possible such $n$ is $n=47$ which is a lesser endpoint than the case above. This confirms claim (b).
    \end{proof}

\section{Main Result}\label{sec:mainresult}

We now have the tools to determine a uniform bound on the vertices visited along the shortest paths in $\Gamma_5$ from $2$ to $a \in S$, for all $a < 100$. Direct computation of the rows $R_i$, discarding all elements greater than this bound, shows that all such $a$ appear within the first 150 rows; this theorem guarantees that in this discarding process we are not forsaking any shorter path. Therefore the first time this computation ``discovers'' each $a$ is indeed the first row in which that $a$ appears.

We begin by establishing how (most) paths in $\Gamma_5$ reach their maximal vertices.

\begin{theorem}\label{thm:howmax}
    Let $P$ be a path in $\Gamma_5$ of length no more than $2\ell$. Denote by $M$ the largest vertex visited along $P$.

    If $P$ visits at least two vertices both before and after it visits $M$, then 
    \begin{enumerate}\renewcommand\labelenumi{(\alph{enumi})}
        \item $M=a^4$ is a fourth power;
        \item $M$ is the penultimate vertex visited along a conveyor belt ending at $a^2$; and
        \item $M \leq  \frac1{16}(5\ell+1)^4.$

    \end{enumerate}
    
\end{theorem}

\begin{proof}
    Denote by $w,x,y,z$ the vertices before and after $M$ in the path: \begin{equation*} \to w \to x \to M \to y \to z \to\end{equation*}

    Since $M$ is the largest vertex visited by $P$, in particular it must be larger than the vertices $x$ and $y$. This guarantees that $M$ is preceded by an up-step and is followed by a down-step, so that $y=\sqrt{M}$ and $x=y-5$:
    \begin{equation*} \to w \to y-5 \stackrel{U}\to M \stackrel{D}\to y \to z\end{equation*}

    If the step $y\to z$ were an up-step, then $z=(y+5)^2 = M + 5y+25 > M$ which would contradict the maximality of $M$. Thus this step must be a down-step, and $z = \sqrt{y} = \sqrt[4]{M}$.

    Replacing $z$ by $a$, we have established claim (a):

    \begin{equation*}
        \to w \to a^2-5 \stackrel{U}\to \; \underbrace{a^4}_{=M} \; \stackrel{D} \to a^2 \stackrel{D}\to a \to 
    \end{equation*}

    By Theorem \ref{thm:uudd}, the step $w \to a^2-5$ cannot be an up-step. This means $w=(a^2-5)^2$ and characterizes the two vertices both before and after $M$ in the path as:
    
    \begin{equation*}
        \to (a^2-5)^2 \stackrel{D}\to a^2-5 \stackrel{U}\to \; \underbrace{a^4}_{=M} \; \stackrel{D} \to a^2 \stackrel{D}\to a \to 
    \end{equation*}
    In particular, it follows that $M$ is visited along a conveyor belt path and not along a ladder. Since the conveyor belt step $(UD)$ containing $M$ is followed by a down-step, i.e. a chute, it is the final step along that conveyor belt, establishing claim (b).

    Assume that the conveyor belt ending at $a^2$ consists of $2i$ steps, that is, we have the path

    \begin{equation*}
        \to a^2-5i \stackrel{(UD)^{i-1}}\to a^2-5 \stackrel{U}\to \; \underbrace{a^4}_{=M} \; \stackrel{D} \to a^2 \stackrel{D}\to a \to 
    \end{equation*}

    We proceed in two cases.
    
{\it Case 1: This conveyor belt is square-to-square.} In this case it must therefore have at least as many steps as that conveyor belt which begins at the adjacent lesser square, $(a-1)^2$.

The conveyor belt $(a-1)^2 \stackrel{(UD)^i} \longrightarrow a^2$ consists of $2i = \frac25\left(a^2-(a-1)^2\right) = \frac25(2a-1)$ steps. We must therefore have
\begin{align*}
    \frac25(2a-1) &\leq 2\ell\\ 
    a &\leq \frac12(5\ell+1).
\end{align*}

{\it Case 2: This conveyor belt is {\it not} square-to-square.} In this case $a^2-5i$ is not a square, so it is preceded by a down-step (since up-steps always end in squares) and, since it is the beginning of its conveyor belt, this down-step is preceded by another down-step (since otherwise the conveyor belt would have begun at a number less than $a^2-5i$).

    \begin{equation*}
        \to (a^2-5i)^4 \stackrel{D}\to(a^2-5i)^2 \stackrel{D}\to a^2-5i \stackrel{(UD)^{i-1}}\to a^2-5 \stackrel{U}\to \; \underbrace{a^4}_{=M} \; \stackrel{D} \to a^2 \stackrel{D}\to a \to 
    \end{equation*}

    Now suppose to the contrary that the conveyor belt has fewer than $\frac1{10}(a^2-1)$ steps, i.e., that $10i < a^2-1$.

    In this case, we have
    \begin{align*}
        (a^2-5i)^2 &= a^4 -10ia^2 + 25i^2 \\ &> a^2(a^2-10i)\\ &> a^2
    \end{align*}
    so that in particular, $(a^2-5i)^4 > a^4 = M$ contradicting the maximality of $M$.

    Thus the conveyor belt has at least $\frac1{10}(a^2-1)$ steps, and in particular \begin{align*}
        \frac1{10}(a^2-1) &\leq 2\ell\\
        a^2 &\leq 20\ell + 1.
    \end{align*}

    Finally, by induction we can establish that for all $\ell \geq 3$ we have $\frac14(5\ell+1)^2 \geq 20\ell+1$. This holds since for $\ell=3$ this inequality takes the form $64 \geq 61$, and if we assume that it holds for arbitrary $\ell$ it follows that
    \begin{align*}
        \frac14\left( 5(\ell+1)+1 \right)^2 &= \frac14(5\ell+1)^2 + \frac14(40\ell+35)\\
        &\geq 20\ell+1 \, + 10\ell + \frac{35}{4}\\
        &\geq 20\ell + 1 \, + 30 + \frac{35}{4} \geq 20\ell + 21 = 20(\ell+1) + 1.
    \end{align*}

    For $\ell < 3$ we lose no generality in saying that $P$ is nevertheless a path of length no more than $6 = 2(3)$ and the above result still holds. Since $P$ necessarily falls into one of either Case 1 or Case 2, the larger of these two quantities bounds $M$ from above.
    
\end{proof}

The above result applies to many paths in $\Gamma_5$. Most notably, it does not cover paths which attain their maximum within one vertex of either endpoint. In the endpoints may be straightforwardly used to bound the maximum so that a complete classification follows:

\begin{corollary}\label{cor:allbound}
    Let $P$ be a path in $\Gamma_5$ that begins at $x$ and ends at $y$ and has length no more than $2\ell$. If $M$ is the largest vertex visited along $P$, then \begin{equation*}
        M \leq \max\left\{\tfrac1{16}(5\ell+1)^4, (x+5)^2, y^2 \right\}.
    \end{equation*}
\end{corollary}

\begin{proof}
    If $P$ contains at least two vertices both before and after $M$, then $M \leq \frac1{16}(5\ell+1)^4$ by Theorem \ref{thm:howmax}.

    If $M$ is attained at the second vertex of $P$, then it must be preceded by an up-step, i.e., $M = Ux = (x+5)^2$.

    If $M$ is attained at the second-to-last vertex of $P$, then it must be followed by a down-step, i.e., $DM = \sqrt{M} = y$.

    If $M$ is attained at the first vertex (respectively last vertex) of $P$, then clearly $M=x$ (respectively $M=y$). Since $x,y \geq 2$ each of these is less than the maximum in the preceding two cases.
\end{proof}

We are now in a position to verify the first appearances of each $a < 100$ within the first $150$ rows of Table \ref{tab:sixrows} using the upper bound in the algorithm from Section \ref{sec:whyhard}.

\begin{theorem}\label{thm:putnambound}
    Let $a\in S$ with $a < 100$. If $a$ appears in row $R_i$ for some $i \leq 150$, then there exists a path $P$ from $2$ to $a$ along which all vertices are no greater than $M=188^4$.
\end{theorem}

\begin{proof}
    This bound follows immediately by setting $\ell=75$ and $x=2$, and stipulating that $y<100$, in Corollary \ref{cor:allbound}.
\end{proof}

\section{Generalizations of Problem A1 and Its Solution}\label{sec:general}

The result of Corollary \ref{cor:allbound} shows how the first appearances of arbitrarily large endpoints, for paths beginning at $2$, may be found by computing the first $2\ell$ rows of $R_i$, discarding entries greater than the appropriate upper bound. 

The methods of this paper may also be applied to more generalized versions of Problem A1. One immediate generalization is to vary the ``root'' of the triangle, the entry on the first row. If $R_1 = \{r\}$ for some $r \geq 2$ that is not a multiple of $5$, Theorem \ref{thm:connected} shows that paths from $r$ to any other endpoint $y\in S$ exist, and the algorithm of Section \ref{sec:whyhard} and upper bound of Corollary \ref{cor:allbound} may be used without adaptation.

Two more interesting generalizations of Problem A1 arise from replacing the number $5$ with another integer $d \geq 2$, and increasing the exponent to investigate higher-power analogues. Here we will show how some known results about the congruence $x^k\equiv y \pmod{d}$, primarily inspired by Somer and K\v{r}\'i\v{z}ek \cite{Somer} and Chou and Shparlinski \cite{ChouShparlinski}, characterize the existence of ``maximal'' solutions of Problem A1.

We begin by defining the generalized problem and establishing criteria for its solution to be called maximal.

\begin{definition}\label{def:generalized_a1}
    For $d,e,r\geq 2$, let $S = S(d,e,r)$ be the smallest set of positive integers such that
    \begin{enumerate}\renewcommand{\labelenumi}{(\alph{enumi})}
        \item $r$ is in $S$,
        \item $n$ is in $S$ whenever $n^e$ is in $S$, and
        \item $(n+d)^e$ is in $S$ whenever $n$ is in $S$.
    \end{enumerate}
    The set $S$ is ``smallest'' in the sense that $S$ is contained in any other such set. We say that $S$ is the \emph{solution of Problem $A1(d,e,r)$}. 
\end{definition}

In this notation, the Putnam problem involved finding the solution of Problem $A1(5,2,2)$, and we call its solution maximal in the following sense.

\begin{definition}
    The solution $S=S(d,e,r)$ of Problem $A1(d,e,r)$ is called \emph{maximal} if it consists of all non-multiples of $d$, i.e., if \begin{equation*}
        S = \bigl\{ n \geq 2 \colon n\not\equiv 0 \pmod{d} \bigr\}.
    \end{equation*}
\end{definition}

In this section we will establish the following result which characterizes which generalized problems have maximal solutions.

\begin{theorem}\label{thm:generalization}
    Problem $A1(d,e,r)$ admits a maximal solution if and only if all of the following conditions hold.
    \begin{enumerate}
        \item $d = p$ is a prime;
        \item $r \not\equiv 0 \pmod{p}$; and
        \item $e$ is a multiple of the radical of $p-1$, that is, given a prime factorization $p-1 = p_1^{n_1}\cdots p_k^{n_k}$ we have $p_1\cdots p_k \mid e$.
    \end{enumerate}
\end{theorem}

The careful reader will correctly suspect that the role played by $p-1$ in condition (3) is owed to Fermat's little theorem (F$\ell$T). Indeed, for exponent (a multiple of) $p-1$ itself the proof is readily adapted from \cite{putnam_2017} and follows immediately from F$\ell$T.

\begin{lemma}\label{lem:specialcase}
    Let $p$ be an odd prime and $r\geq 2$ a non-multiple of $p$. Then Problem $A1(p,p-1,r)$ has a maximal solution.
\end{lemma}
\begin{proof}
    Let $S = S(p,p-1,r)$ be the solution of Problem $A1(p,p-1,r)$. Noting that $\{n\geq 2 \colon n\not\equiv 0\pmod{p}\}$ meets all the conditions of Definition \ref{def:generalized_a1}, we must show that none of its proper subsets do.
    
    Conditions (b) and (c) in Definition \ref{def:generalized_a1} together imply that $S$ is closed under $n \mapsto n + p$. By F$\ell$T we have $R=(r+p)^{p-1}\equiv 1\pmod{p}$ and therefore $R+kp \in S$ for each $k\geq 0$.

    For arbitrary $t \geq 2$ a non-multiple of $p$, we have $t^{p-1} \equiv 1 \pmod{p}$ by F$\ell$T and hence $t^{(p-1)^i} \equiv 1 \pmod{p}$ for each $i\geq 1$. Since this sequence is unbounded above, there exists $m\geq 1,k\geq 0$ for which $t^{(p-1)^m} = R+kp \in S$ and repeated application of Condition (2) thus shows $t \in S$, completing the proof.\end{proof}

In the ``chutes and ladders'' perspective, the maximal solutions in this special case are relatively uninteresting. We henceforth define the generalized directed edges $U(x) = (x+d)^e$ and $D(x) = x^{1/e}$ and observe that F$\ell$T implies the $(p-1)$-th power map modulo $p$ is a star centered at $1$. So it is quickly shown that any two $x,y$ non-multiples of $p$ are connected by a $\{U,D\}$-path since both are at most one $U$-step away from the $1$-belt.

So Theorem \ref{thm:generalization} can be seen as capturing the necessary and sufficient conditions under which every nonzero residue modulo $d$ has a path to the $1$-belt. Conditions (1) and (2) are immediate: the $e$-th power map setwise preserves the $\varphi(d)$-element multiplicative group of units modulo $d$, as well as the $(d-\varphi(d))$-element set of nonunits and a maximal solution makes use of $(d-1)$ residues. That is only possible for prime $d$ (having $\varphi(d)=d-1$ units) and unit $r$ (so that the set $S$ includes the units rather than the nonunits).

The role of Condition (3), then, is to improve upon the ``uninteresting'' case to discover circumstances under which powers \emph{less} than the $(p-1)$-st have connected graphs modulo $p$. The key result derives from Chou and Shparlinski's formula for counting connected components of the $e$-th power map modulo $p$:

\begin{theorem}[\cite{ChouShparlinski}, Thm. 1]
    Write $e = p_1^{n_1} \cdots p_s^{n_s} \geq 2$ and $p-1 = p_1^{r_1}\cdots p_s^{r_s}\, \rho$ where $p_1,\ldots,p_s$ are distinct primes, $n_1,\ldots,n_s \geq 1$ and $r_1,\ldots,r_s \geq 0$ are integers and $\gcd( p_1\cdots p_s,\rho)= 1$. Then the number $N(e,p)$ of connected components in the graph of the $e$-th power map modulo $p$ is
    \begin{equation}\label{eq:connectedcomponents}
    N(e,p) = \sum_{d\mid \rho} \frac{\varphi(d)}{\mathop{\rm ord}_d e}
    \end{equation}
    where $\mathop{\rm ord}_d e$ is the multiplicative order of $e$ modulo $d$.
\end{theorem}

It follows immediately from this result that the graph of the $e$-th power map modulo $p$ is connected, i.e. $N(e,p)=1$, if and only if $\rho = 1$. This holds precisely when all of the prime factors of $p-1$ are also prime factors of $e$, in other words, when $e$ is a multiple of the radical $p_1\cdots p_s$ of $(p-1)$.

This completes the verification of Theorem \ref{thm:generalization}. From this result we have the following corollaries.

\begin{corollary}
    The graph of the $e$-th power map modulo $p$ is connected when there exists $k\geq 1$ such that $p = e^k+1$. 
    
    If $e = q^n$ is a prime power, this graph is connected if and only if there exists $k\geq 1$ such that $p = q^k+1$.

    In particular, the squaring map modulo $p$ is connected if and only if $p = 2^k+1$ is a Fermat prime.
\end{corollary}

\begin{proof}
    In the case where $p-1 = e^k$ it follows from unique factorization that $e$ is a multiple of the radical of $(p-1)$, meeting the criteria of Theorem \ref{thm:generalization}.

    Furthermore when $e = q^n$ for prime $q$ and $n\geq 1$, it follows that $\rho = 1$ if and only if $p-1$ is also a power of $q$, establishing the second claim. Setting $e=q=2$ we have proven the third claim for the squaring map.
\end{proof}

This result suggests that maximal solutions of Problem $A1(p,e,r)$ are in fact quite rare. There are only five known Fermat primes, for example, the finitude of which remains an open question. The prime-power-exponent case of Problem $A1(p^k+1,p^n,r)$ has maximal solutions when $p^k+1$ is a prime; this condition requires in particular that $k$ be a power of two, and hence the prime moduli which may satisfy this condition are a subsequence of the primes one more than a perfect square. Such primes are necessarily found, therefore, within \seqnum{A002496}.

 This corollary settles the question of ``interesting'' maximal solutions to Problem A1 for the case when $p-1$ is a prime power.

\begin{corollary}
    Let $p,q$ be primes such that $p = q^k+1$ for some $k\geq 1$. Then for any $n\geq 1$ and $r\geq 2$, $r\not\equiv 0 \pmod{p}$, the solution to Problem $A1(p,q^n,r)$ is maximal. In particular, in this case there exists a non-star-shaped connected power-map graph modulo $p$ (that is, an ``interesting'' maximal solution).
\end{corollary}

\begin{proof}
    Observe that since $p$ is prime, $p = q^k+1$ in fact necessitates that the exponent $k$ be even. In particular we must have $k\geq 2$, so that $q < p-1$. The graph of the $q$-th power map modulo $p$ is therefore not star-shaped, completing the proof.
\end{proof}

This gives a generalization of the Fermat primes and provides a perhaps-limited number of examples of ``interesting'' maximal solutions when $p-1$ is a prime power. At the other end of the factorization spectrum, consider the case where $p-1$ is squarefree. Here, the only maximal solutions are ``uninteresting'':

\begin{corollary}
    Let $p$ be a prime. There exists $2 \leq e \leq p-2$ for which Problem $A1(p,e,r)$ has a maximal solution if, and only if, $p-1$ is not squarefree.
\end{corollary}

\begin{proof}
    Choose $r\geq 2$, $r\not\equiv 0\pmod{p}$ arbitrarily.
    
    Assume $p-1 = p_1\cdots p_s$ is squarefree, with $p_1,\ldots,p_s$ distinct primes. Since $p-1$ is equal to its radical, by Theorem \ref{thm:generalization} Problem $A1(p,e,r)$ has a maximal solution only when $p_1\cdots p_s = p-1 \mid e$. This rules out maximal solutions for exponents $e \leq p-2$.

    In the case where $p-1$ is not squarefree, let $q$ be a prime for which $q^2 \mid p-1$. Then we may set $e = (p-1)/q \leq p-2$ and every prime factor of $p-1$ is a prime factor of $e$. Hence $\rho=1$ and the solution to Problem $A1(p,e,r)$ is maximal.
\end{proof}

This result shows that interesting maximal solutions exist for, in fact, a majority of prime moduli. Mirsky \cite{Mirsky} has shown that the asymptotic proportion of primes $p$ for which $p-1$ is squarefree is equal to Artin's constant $A \approx 0.374$. These primes, for which all maximal solutions are uninteresting, may also be found as OEIS sequence \seqnum{A039787}.

\subsection{Generalizing the Cutoff Algorithm}

We conclude by showing how to adapt the upper bound of Corollary \ref{cor:allbound} to facilitate a cutoff for determining first appearances/shortest paths in the generalized case.

The row-wise algorithm to output the entries of $S(d,e,r)$ carries over from Definition \ref{def:putnamtriangle}:

\begin{definition}
    For each $i\geq 1$, define $R_i = R_i(d,e,r)$ to be the set of integers defined recursively by
    \begin{itemize}
        \item $R_1 = \{r\}$.
        \item For each $k\geq 1$, if $x \in R_k$ then $(x+d)^e \in R_{k+1}$.
        \item For each $k\geq 1$, if $x^e \in R_k$ then $x \in R_{k+1}$.
    \end{itemize}
\end{definition}

Following the methods of Section \ref{sec:mainresult} we can determine an upper bound on the vertices visited along paths of known starting point, ending point, and length for this generalized problem. This bound can then serve as a cutoff to enable efficient computation of the first appearance of a given number by outputting these rows.

Let $\Gamma(p,e)$ be the directed graph whose vertices are  $V = \cup_{r\geq 2} S(p,e,r)$ and whose $U$- and $D$-edges are given by
\begin{align*}
    U &= \bigl\{ \bigl( n, (n+d)^e \bigr) \colon n \in V \bigr\}\\
    D &= \bigl\{ \bigl( n^e, n \bigr) \colon n \in V \bigr\}.
\end{align*}

The key element in the generalized setting is, as it was in the special case, the prohibition against $UUDD$ subpaths in $\Gamma(p,e)$, analogously to Theorem \ref{thm:uudd}.

\begin{theorem}
    No path in $\Gamma(p,e)$ contains $UUDD$ as a subpath.
\end{theorem}

\begin{proof}
    Suppose to the contrary that there exist $x,y,z\in V$ and a path \begin{equation*} x \stackrel{U}\to (x+d)^e \stackrel{U}\to z \stackrel{D}\to y^e \stackrel{D}\to y.\end{equation*}

    Then it follows that $y^e - (x+d)^e = d$, and since $d$ is positive we have $y > x+d$, and $y^e \geq (x+d+1)^e$. However the latter inequality implies that
    \begin{align*}
        y^e-(x+d)^e &\geq (x+d+1)^e - (x+d)^e\\ &= e(x+d)^{e-1}+\cdots+1 > d,
    \end{align*}
    a contradiction.    
\end{proof}

As before, this limitation on paths enables bounding of the maximum vertex visited, provided there are at least two steps before and after it is attained. 

\begin{corollary}
    Let $P$ be a path in $\Gamma(p,e)$ that begins at $x$, ends at $y$, and has length no more than $2\ell$. If $M$ is the largest vertex visited along $P$, then
    \begin{equation*}
        M \leq \max\left\{\left[1+\left(\frac{d\ell}{e}\right)^{\frac{1}{e-1}} \right]^{e^2},\left(\frac{d\ell}{1-2^{1-e}}\right)^e,(x+d)^e,y^e\right\}.
    \end{equation*}
\end{corollary}

\begin{proof}
That $M \leq (x+d)^e$ and $M\leq y^e$ follows the argument of Corollary \ref{cor:allbound} in the cases where $M$ is attained within the initial or final two steps of the path.

To obtain the remaining estimates, suppose that there are at least two steps both before and after $M$,

$$
w \to x \to M \to y \to z.
$$

The arguments in the proof of Theorem 12 justify why these steps must be, respectively,

$$
w\stackrel{D}\to x \stackrel{U}\to M \stackrel{D}\to y \stackrel{D}\to z
$$

and accordingly we have

$$
w \stackrel{D}\to z^e-d \stackrel{U}\to M=z^{e^2} \stackrel{D}\to z^e \stackrel{D}\to z.
$$

By Theorem \ref{thm:howmax} this path is the tail end of a $2i$-step conveyor belt $(UD)^i$ terminating at $z^e$.

\emph{Case 1:} This conveyor belt contains a lesser $e$-th power.
    
    In this case the belt must contain the belt from $(z-1)^e$ to $z^e$ as a subpath.
    
    That subpath consists of $\frac{2}{d}\bigl( z^e-(z-1)^e \bigr) = \frac{2}{d}\bigl(e(z-1)^{e-1}+\cdots+1)$ steps, and since this is bounded above by $2\ell$ we have in particular

    \begin{align*} \frac{2e}{d}(z-1)^{e-1} &\leq 2\ell\\ z &\leq1+\sqrt[e-1]{d\ell/e}\end{align*}

    from which we conclude $M \leq \bigl(1+\sqrt[e-1]{d\ell/e}\bigr)^{e^2} = O(\ell^e)$.
    
\emph{Case 2:} This conveyor belt does not contain a lesser $e$-th power.
    
    In this case the beginning of the belt $z^e - di$ in particular is not a perfect $e$-th power, so it must be preceded by two down-steps:
    
    $$
    (z^e-di)^{e^2} \stackrel{D}\to (z^e-di)^e \stackrel{D}\to z^e-di \stackrel{(UD)^{i-1}}\longrightarrow z^e-d \stackrel{U}\to M=z^{e^2}
    $$
    
    Since $M$ is maximal we must in particular have $(z^e-di)^{e^2} \leq z^{e^2}$, i.e., $z^e-di \leq z$. This imposes a lower bound on the length of the belt, $i \geq \frac{1}{d}(z^e-z)$, and hence we must have
    
    $$
    \frac1d(z^e-z)\leq \ell.
    $$
    
    However for $z \geq 2$ we have $z^e-z \geq (1-2^{1-e})z^e$. (To see why, note that 
        $z^e-z - (1-2^{1-e})z^e = 2^{1-e}z(z^{e-1}-2^{e-1}) \geq 0.$)

    So we are assured that

    \begin{align*} \frac1d(1-2^{1-e})z^e &\leq \ell\\M=z^{e^2} &\leq\left(\frac{d\ell}{1-2^{1-e}}\right)^e.\end{align*}
    
This exhausts all possibilities and completes the proof.    
\end{proof}

This bound generalizes that of Theorem \ref{thm:putnambound}. For example, by computing the first $100$ rows $R_i(17,4,2)$ and discarding all entries greater than $2357^4 \approx 3.1\times 10^{13}$, we can be assured that the appearance of $7$ on the $84$-th row is indeed its first appearance --- in other words, that the shortest path from $2$ to $7$ in the $\{U,D\}$ graph of the fourth-power map modulo $17$ has $83$ steps.

\bigskip
\hrule
\bigskip

\noindent 2020 {\it Mathematics Subject Classification}: Primary 11Y55, Secondary 11T30

\noindent \emph{Keywords:} squaring map, finite fields,  recreational mathematics

\bigskip
\hrule
\bigskip

\noindent
(Concerned with sequences
\seqnum{A296142}, \seqnum{A321350}, \seqnum{A321351}, and \seqnum{A366552}.)

\bigskip
\hrule
\bigskip

\vspace*{+.1in}
\noindent
Received ;
revised .
Published in {\it Journal of Integer Sequences}, .

\bigskip
\hrule
\bigskip

\bibliographystyle{plain}
\bibliography{refs}

\noindent
Return to
\vskip .1in

\end{document}